\documentclass{article}
\usepackage{amsmath}
\usepackage{graphicx}
\usepackage{amsthm}
\newtheorem{theorem}{Theorem}

\newtheorem{lemma}{Lemma}
\newtheorem{definition}{Definition}
\newtheorem{remark}{Remark}
\newtheorem{example}{Example}
\usepackage{epstopdf}

\usepackage[numbers,sort&compress]{natbib}

\usepackage{subfigure}

\usepackage{algorithmic}
\usepackage{amsfonts,amssymb}
\usepackage[ruled,vlined]{algorithm2e}
\usepackage{indentfirst} 
\title{A class of pseudoinverse-free greedy block nonlinear Kaczmarz methods for nonlinear systems of equations}
\author{Ying Lv, Wendi Bao, Lili Xing, Weiguo Li}
\begin{document}
	\maketitle
	\begin{abstract}
		In this paper, we construct a class of nonlinear greedy average block Kaczmarz methods to solve nonlinear problems without computing the Moore-Penrose pseudoinverse. This kind of methods adopts the average technique of Gaussian Kaczmarz method and combines with the greedy strategy, which greatly reduces the amount of computation. The convergence analysis and numerical experiments of the proposed methods are given. The numerical results show the effectiveness of the proposed methods.
\end{abstract} 
	\section{Introduction}
	Consider to find the roots of system of nonlinear equations
	\begin{equation}  
		f(x)=0,
	\end{equation}
where $f:\mathbb{R}^{n} \to \mathbb{R}^{m}$. We assume throughout that $f(x)=[f_{1}(x),\cdots,f_{m}(x)]^{T}\in \mathbb{R}^{m}$ is acontinuously differentiable vector-valued function, and $x=(x_{1},\cdots,x_{n})^{T}$ is a n-dimensional unknown vector. There exists a solution $x_{\ast}$ such that $f(x_{\ast})=0$. Such nonlinear problems exist in a wide range of practical applications such as machine learning\cite{machinelearning}, differential equations\cite{1}, convex optimization and deep neural networks\cite{DP}.\\
    \indent Recently, the Kaczmarz method\cite{Kaczmarz} has received a lot of attention, due to its simplicity and efficiency. The idea of the classical Kaczmarz method is to project the current point into the solution space given by a row of the coefficient matrix. In 2009, Strohmer and Vershynin proposed a randomized Kaczmarz (RK) method\cite{RK} by selecting the row index in a random order rather than a cycle order and proved that RK has a linear convergence rate. The discovery has sparked renewed interest in the Kaczmarz method\cite{rk1, rk2, rk3, rk4, rk5, rk6, rk7, rk8}. In order to further improve its convergence rate, two kinds of greedy rules are proposed in \cite{maxresidual, maxdistance}, namely maximum residual and maximum distance. Bai and Wu proposed a greedy randomized Kaczmarz method (GRK) \cite{GRK} to accelerate the convergence of the randomized Kaczmarz method. In GRK, in order to annihilate the larger component of the residual preferentially, they proposed the following greedy criterion for selecting the index set.
    \begin{small}
    	\begin{equation}
    		\mathcal{J}_{k} = \left\{ i_{k} \Bigg| \frac{\vert b_{i_{k}} - A_{i_{k}}x_{k} \vert^{2}}{\Vert b-Ax_{k} \Vert_2^{2}} \ge \frac{\Vert A_{i_{k}} \Vert_2^{2}}{2} \left( \frac{1}{\Vert b-Ax_{k} \Vert_2^{2}} \max_{1\le i_{k} \le m}\{\frac{\vert b_{i_{k}}-A_{i_{k}}x_{k} \vert^{2}}{\Vert A_{i_{k}} \Vert_2^{2}}\} +\frac{1}{\Vert A \Vert_F^{2}}\right)\right\}
    		\nonumber
    	\end{equation}
    \end{small}
   
     In order to accelerate the convergence of the classical Kaczmarz method, many researchers studied the block Kaczmarz method\cite{Block, block1, block2}. The idea of the block Kaczmarz method is to use several equations of linear system at the same time in each iteration. For the linear system $Ax=b$, the block method is to use a few rows of the coefficient matrix $A$ at each iteration. The block Kaczmarz method\cite{RBK} can be described as
    \begin{equation}
    		x_{k+1} = x_{k} + A_{\tau_{k}}^{\dagger}(b_{\tau_{k}} - A_{\tau_{k}}x_{k}),     k=0,1,2,\cdots,
    		\nonumber
    \end{equation}
where $A_{\tau_{k}}^{\dagger}$ represents the Moore-Penrose pseudoinverse of the chosen submatrix $A_{\tau_{k}}$ and $\tau_{k}$ is the block row indices. \\
    \indent However, each iteration in the block Kaczmarz method needs to compute the Moore-Penrose pseudoinverse, which ususlly costs expensively. Necoara\cite{Necoara} established a unified framework for the randomized average block Kaczmarz method by taking a convex combination of some updatings as a new direction\cite{FGBK}. The Gaussian Kaczmarz method\cite{GaussianK} can be regarded as another kind of block Kaczmarz method, that is
    \begin{equation}
    	x_{k+1} = x_{k} + \frac{\eta^{T}(b-Ax_{k})}{\Vert A^{T}\eta \Vert_2^{2}}A^{T}\eta,
    	\nonumber
    \end{equation} 
where $\eta$ is a Gaussian vector with mean $0 \in \mathbb{R}^{m}$ and the covariance matrix $I \in \mathbb{R}^{m \times m}$, i.e., $\eta \sim N(0,I)$.\\
    \indent A classical iterative method for solving nonlinear equations is Newton-Raphson method, whose iterative formula is as follows
    \begin{equation}
    	x_{k+1}=x_{k}-(f'(x_{k}))^{\dagger}f(x_{k}).
    \end{equation}
    where $f'(x)=[\nabla f_{1}(x),\cdots,\nabla f_{m}(x)]^{T} \in \mathbb{R}^{m \times n}$ is the Jacobian matrix of $f$ at $x$ and $\nabla f_{i}(x)^{T}$ is its $i$-th row, and $(f'(x_{k}))^{\dagger}$ is the Moore-Penrose pseudoinverse of $f'(x)$. Obviously, this method needs to calculate the entire Jacobian matrix and its Moore-Penrose pseudoinverse which leads to the expensive computation cost. Recently, Wang, Li and Bao generalized the randomized Kaczmarz to the nonlinear case and proposed the nonlinear Kaczmarz method\cite{NRK}, inspired by the Kaczmarz method which only uses one row of the coefficient matrix for each iteration. In \cite{NMR}, Zeng et al. presented a greedy selection strategy as follows:
    \begin{equation}
    	i_{k} = arg\max_{1\le i \le m} \vert f_{i}(x_{k}) \vert^{2}, 
    \end{equation}
    which aims at choosing the maximum component of the residual vector. They also showed that the algorithm with the greedy rule converges faster than NK and NURK methods in both theoretical analysis and experimental results.\\
    \indent In this paper, inspired by \cite{RBCNK} and \cite{MRBNK}, we generalize the pseudoinverse-free block Kaczmarz method for solving linear equations to nonlinear problems and combine greedy rules to further accelerate the convergence of the algorithms. Therefore, we construct a class of pseudoinverse-free greedy block nonlinear Kaczmarz methods with the average technique to avoid computing the pseudoinverse of the Jacobian matrix $f'(x)$ of $f(x)$, which are called the nonlinear greedy average block Kaczmarz (NGABK) method and the maximum residual nonlinear average block Kaczmarz (MRNABK) method. In NGABK method, to avoid calculating the Frobenius norm of the entire Jacobian, we refer to the second greedy rule in \cite{RBCNK}. In MRNABK method, we determine the index set according to the maximum residual rule. The convergence analyses of the two algorithms are given in detail. Numerical experiments show that our proposed methods are more effective than the previous methods. In most cases, the MRNABK method is better than the NGABK method, and both of them are better than several state-of-the-art solvers.\\
    \indent The rest of this paper is organized as follows. In Section \ref{section2}, the notations and preliminaries are provided. In Section \ref{section3}, we provide the two pseudoinverse-free greedy block nonlinear Kaczmarz methods and establish their convergence theorems. The numerical experiments are given in Section \ref{section4}. Finally, we make a summary of the full work in Section \ref{section5}. 
      
	\section{Notations and preliminaries}\label{section2}
	For any matrix $A \in \mathbb{R}^{m \times n}$, we use $\sigma_{max}(A)$, $\sigma_{min}(A)$, $\Vert A \Vert_2$, $\Vert A \Vert_F=\sqrt{\sum_{i=1}^m \sum_{j=1}^n \vert a_{ij} \vert^{2}}$, $A^{\dagger}$ and $A_{\tau}$ to denote the maximum and minimum nonzero singular values of $A$, the spectral norm, the Frobenius norm, the Moore-Penrose pseudoinverse, the row submatrix of matrix $A$ indexed by index set $\tau$. $\vert \tau_{k} \vert$ is the cardinal number of the set $\tau_{k}$. $r_{k}$ denotes the residual vector of the $k$th iteration. For an integer $m \ge 1$, let $[m]:=\{1,\dots,m\}$. For any random variable $\xi$, we use $\mathbb{E}(\xi)$ to denote the expectation of $\xi$.
	\begin{definition}[\cite{NRK}]\label{defi1}
		If every $i \in [m]$ and $\forall x_{1},x_{2} \in \mathbb{R}^{n}$, there exists $\xi_{i} \in [0,\xi)$ satisfying $\xi=\max\limits_{i}\xi_{i} < \frac{1}{2}$ such that
		\begin{equation}
			\vert f_{i}(x_{1})-f_{i}(x_{2})-\nabla f_{i}(x_{1})^{T}(x_{1}-x_{2}) \vert \le \xi_{i}\vert f_{i}(x_{1})-f_{i}(x_{2}) \vert,
		\end{equation}
	then the function f : $\mathbb{R}^{n} \to \mathbb{R}^{m}$ is referred to satisfy the local tangential cone condition.
	\end{definition}
    
    \begin{lemma}[\cite{lihanyu}]\label{yinli1}
        If the function f satisfies the local tangential cone condition, then for $\forall x_{1},x_{2} \in \mathbb{R}^{n}$ and an index subset $\tau \subseteq [m]$, we have
        \begin{equation}
              \Vert f_{\tau}(x_{1}) - f_{\tau}(x_{2}) \Vert_2^{2} \ge \dfrac{1}{1+\xi^{2}}\Vert f'_{\tau}(x_{1})(x_{1}-x_{2}) \Vert_2^2.
        \end{equation}
    \end{lemma}
	\section{Pseudoinverse-free greedy block nonlinear Kaczmarz methods}\label{section3}
	First of all, we briefly introduce the derivation process of Gaussian Kaczmarz method. A random matrix $S \in \mathbb{R}^{m \times q}$ is drawn in an i.i.d. fashion at each iteration, and $q$ is a random variable. For a matrix $B \in \mathbb{R}^{n \times n}$, the $B$-inner product and the induced $B$-norm are defined as follows:
	\begin{equation}
		\langle x,y \rangle_{B}=\langle Bx,y \rangle ,\quad \Vert B \Vert_B = \sqrt{\langle x,x \rangle_{B}} .
		\nonumber
	\end{equation} 
	 \indent For the linear system $Ax=b$, we apply sketch-project technique to this linear system to obtain the following framework\cite{GaussianK},
\begin{equation}
	x_{k+1} = arg\min_{x \in \mathbb{R}^{n}} \Vert x-x_{k} \Vert_B^{2}\quad \rm subject\ to \quad \mathit{S^{T}Ax=S^{T}b}.
	\nonumber
\end{equation}
From the algebraic point of view, the problem can also be written in the following form,
\begin{equation}
	x_{k+1} =\rm\ solution \ of \ \mathit{S^{T}Ax=S^{T}b, \ x=x_{k} + B^{-1}A^{T}Sy}. \label{eq8}
\end{equation}
\indent By substituting the second equation in \eqref{eq8} into the first equation, we have $(S^{T}AB^{-1}A^{T}S)y = S^{T}(b-Ax_{k})$. Notice that all of the solutions $y$ in this system are satisfied \eqref{eq8}. We choose the minimal Euclidean norm solution $y=y_{k}$, which is given by $y_{k} = (S^{T}AB^{-1}A^{T}S)^{\dagger}S^{T}(b-Ax_{k})$. So, we have the following iteration,
\begin{equation}
	x_{k+1} = x_{k} - B^{-1}A^{T}S(S^{T}AB^{-1}A^{T}S)^{\dagger}S^{T}(Ax_{k}-b).\label{eq9}
\end{equation}
	\indent When $S$ is a Gaussian vector with mean 0 $\in \mathbb{R}^{m}$ and a positive definite covariance matrix $\Sigma \in \mathbb{R}^{m \times m}$, specifically, $S = \zeta \sim N(0,\Sigma)$, we plugging $S$ into \eqref{eq9} gives us the following result,
	\begin{equation}
		x_{k+1} = x_{k} - \frac{\zeta^{T}(Ax_{k}-b)}{\zeta^{T}AB^{-1}A^{T}\zeta}B^{-1}A^{T}\zeta.\label{eq10}
	\end{equation}
Let $B=I$, and choose $\Sigma = I$ so that $S=\eta \sim N(0,I)$. Then \eqref{eq10} has the form
\begin{equation}
	x_{k+1} = x_{k} - \frac{\eta^{T}(Ax_{k}-b)}{\Vert A^{T}\eta \Vert_2^{2}}A^{T}\eta
	\nonumber
\end{equation}
	which is the Gaussian Kaczmarz (GK) method.\\
	\indent Next, the nonlinear case is similar to the derivation above. By applying the sketch-project technique to the Newton-Raphson(NR) method, we obtain the Sketched Newton-Raphson(SNR) method\cite{SNR}. We project the current point $x_{k}$ onto the solution space of the Newton system
	\begin{equation}
		x_{k+1} = arg\min_{x \in \mathbb{R}^{n}}\Vert x-x_{k} \Vert^{2} \quad s.\ t. \quad f'(x_{k})(x-x_{k}) = f(x_{k}),
\nonumber	
\end{equation}
which is the classic Newton-Raphson(NR) method. Similar to the NR method, we project the current point $x_{k}$ onto the solution space of the sketched Newton system
\begin{equation}
	x_{k+1} = arg\min_{x \in \mathbb{R}^{n}}\Vert x-x_{k} \Vert^{2} \quad s.\ t. \quad S_{k}^{T}f'(x_{k})(x-x_{k}) = S_{k}^{T}f(x_{k}),
	\nonumber
\end{equation}
in which $S_{k}^{T} \in \mathbb{R}^{\tau \times m}$ is the sketching matrix. Specifically, the formula of SNR method can be written as follows,
\begin{equation}
	x_{k+1} = x_{k} - (f'(x_{k}))^{T}S_{k}(S_{k}^{T}f'(x_{k})(f'(x_{k}))^{T}S_{k})^{\dagger}S_{k}^{T}f(x_{k}).\label{eq14}
\end{equation}
	In \eqref{eq14}, we set $S=\eta \sim N(0,I)$. Then \eqref{eq14} has the form
	\begin{equation}
			x_{k+1}=x_{k}-\frac{\eta_{k}^{T}f(x_{k})}{\Vert f^{'}(x_{k})^{T}\eta_{k} \Vert_2^{2}}f'(x_{k})^{T}\eta_{k},
			\nonumber
	\end{equation}
	which is the iterative formula in this paper.
	
	 Based on the residual-distance capped nonlinear Kaczmarz (RD-CNK) method\cite{RBCNK}, we present a class of pseudoinverse-free greedy block nonlinear Kaczmarz methods for solving nonlinear systems. The block indices $\mathcal{I}_{k}$ is chosen by
	\begin{equation}
		\mathcal{I}_{k}=\{ i\mid \vert f_{i}(x_{k}) \vert^{2} \ge \delta_{k}\Vert f(x_{k}) \Vert_2^{2} \Vert  \nabla f_{i}(x_{k}) \Vert_2^{2} \}
		\nonumber
	\end{equation}
with
    \begin{equation}
    	\delta_{k}=\dfrac{1}{2}\left( \frac{1}{\Vert f(x_{k}) \Vert_2^{2}}\max\limits_{i\in [m]}\frac{\vert f_{i}(x_k) \vert^{2}}{\Vert  \nabla f_{i}(x_{k}) \Vert_2^{2}} + \frac{1}{\Vert f'(x_{k}) \Vert_F^{2}} \right).
    	\nonumber
    \end{equation}
    \indent Considering that the above selection rule requires the calculation of the Frobenius norm of the entire Jacobian matrix, this will result in a large amount of computation. Therefore, we choose the block indices $\tau_{k}$ by
    \begin{equation}
    	\tau_{k}=\{i\vert f_{i}(x_{k}) \vert^{2}\geq \delta_{k}\Vert f(x_{k}) \Vert_2^{2}\}
    	\nonumber
    \end{equation}
with
    \begin{equation}
    	\delta_{k}=\frac{1}{2}(\frac{\max\limits_{i\in m}\lvert f_{i}(x_{k}) \rvert^{2}}{\Vert f(x_{k}) \Vert_2^{2}}+\frac{1}{m}).
    	\nonumber
    \end{equation}
\indent Given an initial guess vector $x_{0} \in \mathbb{R}^{n}$, the nonlinear greedy average block Kaczmarz  method is described in Algorithm \ref{alg:algorithm1}.
	\begin{algorithm}[t]
	\caption{The NGABK algorithm}
	\label{alg:algorithm1}
	\begin{algorithmic}[1]
		\REQUIRE{The initial estimate $x_{0}\in \mathbb{R}^{n}$.}
		\STATE $\mathbf{for}$ $k$=1,2,$\cdots$ until convergence, $\mathbf{do}$
		\STATE Compute
		\begin{equation}
			\delta_{k}=\frac{1}{2}\big(\frac{\max\limits_{i\in m}\lvert f_{i}(x_{k}) \rvert^{2}}{\Vert f(x_{k}) \Vert_2^{2}}+\frac{1}{m}\big)
		\end{equation}
		\STATE Determine the index subset
		\begin{equation}
			\tau_{k}=\{i\big|\vert f_{i}(x_{k}) \vert^{2}\geq \delta_{k}\Vert f(x_{k}) \Vert_2^{2}\}
		\end{equation}
		\STATE Compute
		\begin{equation}
			\eta_{k}=\sum_{i \in \tau_{k}}(-f_{i}(x_{k}))e_{i}
		\end{equation}
		\STATE Set
		\begin{equation}\label{gongshi}
			x_{k+1}=x_{k}-\frac{\eta_{k}^{T}f(x_{k})}{\Vert f^{'}(x_{k})^{T}\eta_{k} \Vert_2^{2}}f'(x_{k})^{T}\eta_{k}
		\end{equation}
		\STATE $\mathbf{end \ for}$
		\BlankLine
	\end{algorithmic}
\end{algorithm}
    \begin{remark}
    	The method is well defined as the index set $\tau_{k}$ is always nonempty.This is because
    	\begin{equation}
    		\max\limits_{i\in [m]}\vert f_{i}(x_{k}) \vert^{2} \ge \sum_{i=1}^{m}\frac{\vert f_{i}(x_{k}) \vert^{2}}{m} = \frac{\Vert f(x_{k}) \Vert_2^{2}}{m}
    		\nonumber
    	\end{equation}
    and then
    \begin{equation}
    	\vert f_{i_{k}}(x_{k}) \vert^{2} = \max\limits_{i\in [m]}\vert f_{i}(x_{k}) \vert^{2} \ge \frac{1}{2} \left( \max\limits_{i\in [m]}\vert f_{i}(x_{k}) \vert^{2} + \frac{\Vert f(x_{k}) \Vert_2^{2}}{m}  \right)
    	\nonumber
    \end{equation}
implies $i_{k} \in \tau_{k}$.
    \end{remark}
\indent According to the maximum residual rule, we establish the maximum residual nonlinear average block Kaczmarz in Algorithm \ref{alg:algorithm2}. In this method, $\varrho$ is the relaxation parameter, which can be determined in numerical experiments. It is obvious that in the $k$-th iteration of the Algorithm \ref{alg:algorithm2}, the set $\tau_{k}$ is also non-empty. This is because\\
\begin{equation}
	\vert f_{i_{k}}(x_{k}) \vert = \max_{1\le i \le m}\vert f_{i}(x_{k}) \vert.
	\nonumber
\end{equation}
That is to say, the largest residual component $i_{k}$ is always in the set $\tau_{k}$.\\
	\begin{algorithm}[t]
	\caption{The MRNABK algorithm}
	\label{alg:algorithm2}
	\begin{algorithmic}[1]
		\REQUIRE{The initial estimate $x_{0}\in \mathbb{R}^{n}$.}
		\STATE $\mathbf{for}$ $k$=1,2,$\cdots$ until convergence, $\mathbf{do}$
		\STATE Determine the index subset
		\begin{equation}
			\tau_{k}=\{i \big| \vert f_{i}(x_{k}) \vert^{2}\geq \varrho \max_{1\le i \le m}\vert f(x_{k}) \vert^{2}\}
		\end{equation}
		\STATE Compute
		\begin{equation}
			\eta_{k}=\sum_{i \in \tau_{k}}(-f_{i}(x_{k}))e_{i}
		\end{equation}
		\STATE Set
		\begin{equation}\label{update formula}
			x_{k+1}=x_{k}-\frac{\eta_{k}^{T}f(x_{k})}{\Vert f^{'}(x_{k})^{T}\eta_{k} \Vert_2^{2}}f'(x_{k})^{T}\eta_{k}
		\end{equation}
		\STATE $\mathbf{end \ for}$
		\BlankLine
	\end{algorithmic}
\end{algorithm}
\indent From Algorithm \ref{alg:algorithm2}, we can see that the larger components of the residual are eliminated preferentially, which greatly improves the efficiency of the algorithm. We establish its convergence theorem in Section \ref{section4}.
	\section{Convergence analysis}\label{section4}
	\begin{lemma}\label{yinli2}
		If the function f satisfies the local tangential cone condition, then for $i \in \left[ m \right]$, $\xi = \max\limits_{i\in [m]} \xi_{i} < \dfrac{1}{2}$, $\forall x_{1},x_{2} \in \mathbb{R}^{n}$ and the updating formula (\ref{gongshi}), we have
		\begin{equation}
			\Vert x_{k+1}-x_{\ast} \Vert_2^{2} \le \Vert x_{k}-x_{\ast} \Vert_2^{2} - (1-2\xi)\frac{\vert \eta_{k}^{T}f(x_{k}) \vert^{2}}{\Vert f^{'}(x_{k})^{T}\eta_{k} \Vert_2^{2}}.
		\end{equation}
	\end{lemma}
\begin{proof}
	From the updating fomula \eqref{update formula}, we have
	\begin{equation}
		\begin{aligned}
		\Vert x_{k+1}-x_{\ast} \Vert_2^{2} &= \Vert x_{k} - \frac{\eta_{k}^{T}f(x_{k})}{\Vert f^{'}(x_{k})^{T}\eta_{k} \Vert_2^{2}}f'(x_{k})^{T}\eta_{k} - x_{\ast} \Vert_2^{2}\\
		&= \Vert x_{k}-x_{\ast} \Vert_2^{2}-2\left< \frac{\eta_{k}^{T}f(x_{k})}{\Vert f^{'}(x_{k})^{T}\eta_{k} \Vert_2^{2}}f'(x_{k})^{T}\eta_{k},x_{k}-x_{\ast} \right> + \frac{\vert \eta_{k}^{T}f(x_{k}) \vert^{2}}{\Vert f^{'}(x_{k})^{T}\eta_{k} \Vert_2^{2}}\\
		\nonumber
		\end{aligned}
	\end{equation}
According to the definition of $\eta_{k}$, we have 
    \begin{equation}
	\begin{aligned}
		&\Vert x_{k+1}-x_{\ast} \Vert_2^{2}\\
		& = \Vert x_{k}-x_{\ast} \Vert_2^{2} - 2\frac{\eta_{k}^{T}f(x_{k})}{\Vert f^{'}(x_{k})^{T}\eta_{k} \Vert_2^{2}}(-\sum_{i \in \tau_{k}}f_{i}(x_{k})e_{i}^{T}(f'(x_{k})))(x_{k}-x_{\ast})\\
		&+ \frac{\vert \eta_{k}^{T}f(x_{k}) \vert^{2}}{\Vert f^{'}(x_{k})^{T}\eta_{k} \Vert_2^{2}}\\
		&= \Vert x_{k}-x_{\ast} \Vert_2^{2} + 2\frac{\eta_{k}^{T}f(x_{k})}{\Vert f^{'}(x_{k})^{T}\eta_{k} \Vert_2^{2}}(\sum_{i \in \tau_{k}}f_{i}(x_{k})\nabla f_{i}(x_{k})^{T}))(x_{k}-x_{\ast}) + \\
		&\frac{\vert \eta_{k}^{T}f(x_{k}) \vert^{2}}{\Vert f^{'}(x_{k})^{T}\eta_{k} \Vert_2^{2}}\\
		&= \Vert x_{k}-x_{\ast} \Vert_2^{2} - 2\frac{\eta_{k}^{T}f(x_{k})}{\Vert f^{'}(x_{k})^{T}\eta_{k} \Vert_2^{2}}\sum_{i \in \tau_{k}}f_{i}(x_{k})(f_{i}(x_{k})-f_{i}(x_{\ast})-\\
		&\nabla f_{i}(x_{k})^{T}(x_{k}-x_{\ast}))+ 2\frac{\eta_{k}^{T}f(x_{k})}{\Vert f^{'}(x_{k})^{T}\eta_{k} \Vert_2^{2}}\sum_{i \in \tau_{k}}f_{i}^{2}(x_{k})+\frac{\vert \eta_{k}^{T}f(x_{k}) \vert^{2}}{\Vert f^{'}(x_{k})^{T}\eta_{k} \Vert_2^{2}}.\\
		\nonumber
	\end{aligned}
    \end{equation}
From Definition \ref{defi1}, we have
    \begin{equation}
    	\begin{aligned}
    		\Vert x_{k+1}-x_{\ast} \Vert_2^{2} &\le \Vert x_{k}-x_{\ast} \Vert_2^{2} + 2\frac{\vert \eta_{k}^{T}f(x_{k}) \vert^{2}}{\Vert f^{'}(x_{k})^{T}\eta_{k} \Vert_2^{2}}\xi - \frac{\vert \eta_{k}^{T}f(x_{k}) \vert^{2}}{\Vert f^{'}(x_{k})^{T}\eta_{k} \Vert_2^{2}}\\
    		&= \Vert x_{k}-x_{\ast} \Vert_2^{2} - (1-2\xi)\frac{\vert \eta_{k}^{T}f(x_{k}) \vert^{2}}{\Vert f^{'}(x_{k})^{T}\eta_{k} \Vert_2^{2}}.
    		\nonumber
    	\end{aligned}
    \end{equation}
\end{proof}
\begin{theorem}
	If the nonlinear function f satisfies the local tangential cone condition given in Definition \ref{defi1}, $\xi = \max\limits_{i\in [m]} \xi_{i} < \dfrac{1}{2}$, $f(x_{\ast}) = 0$, and $f'(x)$ is a full column rank matrix, then the iterations of the NGABK method in Algorithm \ref{alg:algorithm1} satisfy
	\begin{equation}
		\Vert x_{k+1}-x_{\ast} \Vert_2^{2} \le \left( 1-\dfrac{1-2\xi}{1+\xi^{2}}\dfrac{\delta_{k}\vert \tau_{k} \vert \sigma_{min}^{2}(f'(x_{k}))}{\sigma_{max}^{2}(f'_{\tau_{k}}(x_{k}))} \right)\Vert x_{k}-x_{\ast} \Vert_2^{2}.
	\end{equation}
\end{theorem}
\begin{proof}
		Let $E_{k} \in \mathbb{R}^{m \times \vert \tau_{k} \vert}$ be the matrix whose columns consist of all the vectors $e_{i} \in \mathbb{R}^{m}$ with $i \in \tau_{k}$. Denote $f'_{\tau_{k}}(x_{k}) = E_{k}^{T}f'(x_{k})$, $\hat{\eta_{k}} = E_{k}^{T}\eta_{k}$, then
\begin{equation}
	\Vert \hat{\eta_{k}} \Vert_2^{2} = \eta_{k}^{T}E_{k}E_{k}^{T}\eta_{k} = \Vert \eta_{k} \Vert_2^{2} = \sum_{i \in \tau_{k}}\vert f_{i}(x_{k}) \vert^{2}
\end{equation}
and
		\begin{equation}
			\begin{aligned}
			\Vert f'(x_{k})^{T}\eta_{k} \Vert_2^{2} &= \eta_{k}^{T}f'(x_{k})f'(x_{k})^{T}\eta_{k} = \hat{\eta_{k}}^{T}E_{k}^{T}f'(x_{k})f'(x_{k})^{T}E_{k}\hat{\eta_{k}} \\
			&= \hat{\eta_{k}}^{T} f'_{\tau_{k}}(x_{k}) f'_{\tau_{k}}(x_{k})^{T} \hat{\eta_{k}} = \Vert f'_{\tau_{k}}(x_{k})\hat{\eta_{k}} \Vert_2^{2}.
			\nonumber
			\end{aligned}
		\end{equation}
Therefore, we have
\begin{equation}
	\Vert f'_{\tau_{k}}(x_{k})^{T}\hat{\eta_{k}} \Vert_2^{2} = \hat{\eta_{k}}^{T} f'_{\tau_{k}}(x_{k}) f'_{\tau_{k}}(x_{k})^{T} \hat{\eta_{k}} \le \sigma_{max}^{2}(f'_{\tau_{k}}(x_{k})) \Vert \hat{\eta_{k}} \Vert_2^{2},
	\nonumber
\end{equation}
where $\sigma_{max}(f'_{\tau_{k}}(x_{k}))$ is the largest singular value of submatrix $f'_{\tau_{k}}(x_{k})$ of the Jacobian matrix $f'(x_{k})$.
From the definition of $\eta_{k}$, we have
\begin{equation}
	\begin{aligned}
		\eta_{k}^{T}(-f(x_{k})) &= \left(\sum_{i \in \tau_{k}} (- f_{i}(x_{k}))e_{i}^{T}\right)(-f(x_{k})) \\
		&= \sum_{i \in \tau_{k}}f_{i}(x_{k})e_{i}^{T}(f(x_{k})) \\
		&= \sum_{i \in \tau_{k}} \vert f_{i}(x_{k}) \vert^{2} \\
		&= \Vert \hat{\eta_{k}} \Vert_2^{2}.
		\nonumber
	\end{aligned}
\end{equation}
From the definition of $\tau_{k}$, we have
\begin{equation}
	\begin{aligned}
		\frac{\vert \eta_{k}^{T}(-f(x_{k})) \vert^{2}}{\Vert f'(x_{k})^{T}\eta_{k} \Vert_2^{2}} &= \frac{\left( \sum_{i \in \tau_{k}} \vert f_{i}(x_{k}) \vert^{2}\right) \Vert \hat{\eta_{k}} \Vert_2^{2}}{\Vert f'_{\tau_{k}}(x_{k})^{T} \hat{\eta_{k}} \Vert_2^{2}}\\
		&\ge \frac{\sum_{i \in \tau_{k}} \vert f_{i}(x_{k}) \vert^{2}}{\sigma_{max}^{2}(f'_{\tau_{k}}(x_{k}))}\\
		&\ge \frac{ \sum_{i \in \tau_{k}}\delta_{k} \Vert f(x_{k}) \Vert_2^{2}}{\sigma_{max}^{2}(f'_{\tau_{k}}(x_{k}))}\\
		&= \frac{\delta_{k} \vert \tau_{k} \vert}{\sigma_{max}^{2}(f'_{\tau_{k}}(x_{k}))} \Vert f(x_{k}) - f(x_{\ast}) \Vert_2^{2}.
		\nonumber
	\end{aligned}
\end{equation}
\begin{equation}
	\begin{aligned}
		\frac{\vert \eta_{k}^{T}f(x_{k}) \vert^{2}}{\Vert f'(x_{k})^{T}\eta_{k} \Vert_2^{2}} &\ge \frac{\delta_{k} \vert \tau_{k} \vert}{\sigma_{max}^{2}(f'_{\tau_{k}}(x_{k}))} \cdot \frac{1}{1+\xi^{2}} \Vert f'(x_{k})(x_{k}-x_{\ast}) \Vert_2^{2}\\
		&\ge \frac{\delta_{k} \vert \tau_{k} \vert}{\sigma_{max}^{2}(f'_{\tau_{k}}(x_{k}))} \cdot \frac{1}{1+\xi^{2}} \cdot \sigma_{min}^{2}(f'(x_{k})) \Vert x_{k} - x_{\ast} \Vert_2^{2}.
		\nonumber
	\end{aligned}
\end{equation}
Further, using Lemma \ref{yinli2}, we can obtain
\begin{equation}
	\begin{aligned}
		\Vert x_{k+1} - x_{\ast} \Vert_2^{2} &\le \Vert x_{k} - x_{\ast} \Vert_2^{2} - \frac{(1-2\xi)\delta_{k}\vert \tau_{k} \vert \sigma_{min}^{2}(f'(x_{k}))}{(1+\xi^{2})\sigma_{max}^{2}(f'_{\tau_{k}}(x_{k}))} \Vert x_{k} - x_{\ast} \Vert_2^{2}\\
		&=\left( 1 - \frac{(1-2\xi)\delta_{k}\vert \tau_{k} \vert \sigma_{min}^{2}(f'(x_{k}))}{(1+\xi^{2})\sigma_{max}^{2}(f'_{\tau_{k}}(x_{k}))}\right)\Vert x_{k} - x_{\ast} \Vert_2^{2}.
	\end{aligned}\notag
\end{equation}
So, the convergence of NGABK is proved.
\end{proof}
\begin{remark}
	\rm Since $\Vert x_{k+1} - x_{\ast} \Vert_2^{2} \ge 0$, we have $\rho_{NGABK} \ge 0$. In addition, we have $\sigma_{max}^{2}(f'_{\tau_{k}}(x_{k})) = \Vert f'_{\tau_{k}}(x_{k}) \Vert_2^{2} \le \Vert f'(x_{k}) \Vert_F^{2}$ and  $\delta_{k} = \frac{1}{2}(\frac{\max\limits_{i\in m}\lvert f_{i}(x_{k}) \rvert^{2}}{\Vert f(x_{k}) \Vert_2^{2}}+\frac{1}{m}) \ge \frac{1}{m}$, so 
	\begin{equation}
		\rho_{NGABK} = 1-\dfrac{1-2\xi}{1+\xi^{2}}\dfrac{\delta_{k}\vert \tau_{k} \vert \sigma_{min}^{2}(f'(x_{k}))}{\Vert f'_{\tau_{k}}(x_{k}) \Vert_2^{2}} < 1-\frac{1-2\xi}{(1+\xi)^2}\dfrac{\sigma_{min}^{2}(f'(x_{k}))}{m \Vert f'(x_{k})\Vert_F^{2}} = \rho_{NRK}. \nonumber
	\end{equation}
This shows that the convergence factor of our method is strictly smaller than that of NRK method.
\end{remark}
Now, we give the convergence theorem of Algorithm \ref{alg:algorithm2}.
\begin{theorem}
	If the nonlinear function f satisfies the local tangential cone condition given in Definition 1, $\xi = \max\limits_{i\in [m]} \xi_{i} < \dfrac{1}{2}$, $f(x_{\ast}) = 0$, and $f'(x)$ is a full column rank matrix, then the iterations of the MRNABK method in Algorithm \ref{alg:algorithm2} satisfy
		\begin{equation}
		\Vert x_{k+1}-x_{\ast} \Vert_2^{2} \le \left( 1-\dfrac{1-2\xi}{1+\xi^{2}}\dfrac{\varrho \vert \tau_{k} \vert \sigma_{min}^{2}(f'(x_{k}))}{m\sigma_{max}^{2}(f'_{\tau_{k}}(x_{k}))} \right)\Vert x_{k}-x_{\ast} \Vert_2^{2}.
	\end{equation}
\end{theorem}
\begin{proof}
	Following an analogous proof process to the NGABK method, we can get the following formula
	\begin{equation}
		\begin{aligned}
			\frac{\vert \eta_{k}^{T}(-f(x_{k})) \vert^{2}}{\Vert f'(x_{k})^{T}\eta_{k} \Vert_2^{2}} &= \frac{\left( \sum_{i \in \tau_{k}} \vert f_{i}(x_{k}) \vert^{2}\right) \Vert \hat{\eta_{k}} \Vert_2^{2}}{\Vert f'_{\tau_{k}}(x_{k})^{T} \hat{\eta_{k}} \Vert_2^{2}}\\
			&\ge \frac{\sum_{i \in \tau_{k}} \vert f_{i}(x_{k}) \vert^{2}}{\sigma_{max}^{2}(f'_{\tau_{k}}(x_{k}))}\\
			&\ge \frac{ \sum_{i \in \tau_{k}}\varrho \max \limits_{1\le i \le m}\vert f_{i}(x_{k}) \vert^{2}}{\sigma_{max}^{2}(f'_{\tau_{k}}(x_{k}))}\\
			&\ge \frac{\varrho \vert \tau_{k} \vert}{m\sigma_{max}^{2}(f'_{\tau_{k}}(x_{k}))} \Vert f(x_{k}) - f(x_{\ast}) \Vert_2^{2}.
			\nonumber
		\end{aligned}
	\end{equation}
\indent The second inequality follows from the definition of $\tau_{k}$. Using $\max \limits_{1\le i \le m}\vert f_{i}(x_{k}) \vert^{2} \ge \frac{1}{m}\Vert f(x_{k})\Vert_2^{2}$, we can get the third inequality.\\
\indent From Lemma \ref{yinli1}, it follows that
\begin{equation}
	\begin{aligned}
		\frac{\vert \eta_{k}^{T}f(x_{k}) \vert^{2}}{\Vert f'(x_{k})^{T}\eta_{k} \Vert_2^{2}} &\ge \frac{\varrho \vert \tau_{k} \vert}{m\sigma_{max}^{2}(f'_{\tau_{k}}(x_{k}))} \cdot \frac{1}{1+\xi^{2}} \Vert f'(x_{k})(x_{k}-x_{\ast}) \Vert_2^{2}\\
		&\ge \frac{\varrho \vert \tau_{k} \vert}{m\sigma_{max}^{2}(f'_{\tau_{k}}(x_{k}))} \cdot \frac{1}{1+\xi^{2}} \cdot \sigma_{min}^{2}(f'(x_{k})) \Vert x_{k} - x_{\ast} \Vert_2^{2}.
		\nonumber
	\end{aligned}
\end{equation}
Further, using Lemma \ref{yinli2}, we can obtain
\begin{equation}
	\begin{aligned}
		\Vert x_{k+1} - x_{\ast} \Vert_2^{2} &\le \Vert x_{k} - x_{\ast} \Vert_2^{2} - \frac{(1-2\xi)\varrho\vert \tau_{k} \vert \sigma_{min}^{2}(f'(x_{k}))}{(1+\xi^{2})m\sigma_{max}^{2}(f'_{\tau_{k}}(x_{k}))} \Vert x_{k} - x_{\ast} \Vert_2^{2}\\
		&=\left( 1 - \frac{(1-2\xi)\varrho \vert \tau_{k} \vert \sigma_{min}^{2}(f'(x_{k}))}{(1+\xi^{2})m\sigma_{max}^{2}(f'_{\tau_{k}}(x_{k}))}\right)\Vert x_{k} - x_{\ast} \Vert_2^{2}.
		\nonumber
	\end{aligned}
\end{equation}
\indent So, the convergence of MRNABK is proved.
\end{proof}
	\section{Numerical examples}\label{section5}
	In this section, we mainly compare the efficiency of our new methods with NRK, the residual-distance capped nonlinear Kaczmarz (RD-CNK) and the residual-based block capped nonlinear Kaczmarz (RB-CNK) for solving the nonlinear systems of equations in the iteration numbers (denoted as ‘IT’) and computation time (denoted as ‘CPU’). RD-CNK and NRK are based on a single sample. RB-CNK is based on multi-sampling and uses the following iteration scheme: 
	\begin{equation}
		x_{k+1} = x_{k} - (f'_{\mathcal{I}_{k}}(x_{k}))^{\dagger} f_{\mathcal{I}_{k}}(x_{k}),
		\nonumber
	\end{equation}
where $\mathcal{I}_{k}$ is the selected index subset. The target block in MRNABK is caculated by
\begin{equation}
	\mathcal{I}_{k} = \{ i_{k} \vert \ \vert r_{i_{k}} \vert^{2} \ge \varrho \max _{1\le i \le m}\vert r(i) \vert^{2} \}
	\nonumber
\end{equation} 
where the parameter $\varrho$ in the experiment is set to 0.1.\\
\indent In the numerical experiment, the IT and CPU are the average of the results of 10 times repeated runs of the corresponding method. All experiments are terminated when the number of iterations exceeds 200,000 or $\Vert f(x_{k}) \Vert_2^{2} < 10^{-6}$. Our experiment is implemented on MATLAB (version R2018b).\\
\begin{example}
	\rm In this example, we consider the following equations
	\begin{equation}
		F_{i}(x) = x_{i} - (1 - \dfrac{c}{2N} \sum_{j=1}^{N} \frac{\mu_{i}x_{j}}{\mu_{i} + \mu_{j}})^{-1}.
		\nonumber
	\end{equation}
	\indent The system of the equations is called H-equation, which is usually used to solve the problem of outlet distribution in radiation transmission. In this problem, N represents the number of equations and $\mu_{i} = (i-\frac{1}{2})/N$. We set $x_{0}$ be zero vectors and $c=0.9$. First of all, We test the value of parameter $\varrho$. In Table \ref{tab1}, we observe that the computation time of MRNABK is relatively low in most cases, when $\varrho=0.1, 0.2, 0.3$. When the number of equations is fixed, we can find that the larger $\varrho$ is, the longer the calculation time of MRNABK will be. Next, we test the performance of our methods and other methods. The results of numerical experiments are listed in Table \ref{tab2} and Table \ref{tab3}. As can be seen from Fig. \ref{fig:1}, RD-CNK and RB-CNK methods based on single sampling are significantly slower than RB-CNK and NGABK methods based on multiple sampling. NGABK method is slightly better than RB-CNK method in iteration time and iteration times from Fig. \ref{fig:2}. From Table \ref{tabMRNABKIT}, MRNABK converges faster than NGABK in terms of computing time and iteration steps.
	
	\begin{table}
		\centering
		\caption{CPU of MRNABK for Example 1 with $c=0.9$, $x_{0}=0$ and different $\varrho$}
		\label{tab1}
		\begin{tabular}{ccccccc}
			\hline
			$\varrho$& 0.1& 0.3& 0.5& 0.7& 0.8& 0.9\\
			\hline
			$m=50$& 0.018& 0.0260& 0.0218& 0.0282& 0.0334& 0.0441\\
			$m=100$& 0.0958& 0.0863& 0.1607& 0.1440& 0.2222& 0.1946\\
			$m=500$& 1.2712& 1.2321& 1.3973& 1.5922& 1.7479& 2.0144\\
			$m=1000$& 3.9016& 4.1044& 4.8689& 5.1860& 5.9891& 6.7390\\
			$m=1500$& 8.2839& 8.9858& 10.4272& 11.0643& 11.8178& 13.9577\\
			\hline
		\end{tabular}	
	\end{table}
	
	\begin{table}
		\centering
		\caption{IT comparison of NRK, RD-CNK, NGABK, RB-CNK, MRNABK}
		\label{tab2}
		\begin{tabular}{cccccc}
			\hline
			$n$& NRK& RD-CNK& NGABK& RB-CNK& MRNABK\\
			\hline
			$x_{0} = zero(50,1)$& 970& 864& 70& 62& 21\\
			$x_{0} = zero(100,1)$& 2022& 1814& 66& 66& 21\\
			$x_{0} = zero(300,1)$& 6518& 5838& 72& 76& 24\\
			$x_{0} = zero(500,1)$& 11239& 10027& 78& 81& 24\\
			\hline
		\end{tabular}	
	\end{table}
	
	\begin{table}
		\centering
		\caption{CPU comparison of NRK, RD-CNK, NGABK, RB-CNK, MRNABK}
		\label{tab3}
		\begin{tabular}{cccccc}
			\hline
			$n$& NRK& RD-CNK& NGABK& RB-CNK& MRNABK\\
			\hline
			$x_{0} = zero(50,1)$& 0.2646& 0.5149& 0.0640& 0.0989& 0.0593\\
			$x_{0} = zero(100,1)$& 0.4751& 1.6045& 0.1064& 0.1680& 0.0754\\
			$x_{0} = zero(300,1)$& 4.3069& 26.9104& 0.6370& 0.8359& 0.4770\\
			$x_{0} = zero(500,1)$& 12.2161& 95.7063& 1.4577& 1.9005& 1.0813\\
			\hline
		\end{tabular}	
	\end{table}
	
	\begin{table}
		\centering
		\caption{IT and CPU comparison of MRNABK and NGABK with $\varrho=0.1$ and $c=0.9$}
		\label{tabMRNABKIT}
		\begin{tabular}{ccccc}
			\hline
			$n$& MRNABK(IT)& NGABK(IT)& MRNABK(CPU)& NGABK(CPU)\\
			\hline
			$50 \times 50$& 21& 70& 0.0433& 0.0742\\
			$100 \times 100$& 21& 66& 0.1067& 0.1631\\
			$500 \times 500$& 24& 78& 1.6064& 1.1954\\
			$1000 \times 1000$& 25& 78& 3.8946& 5.0264\\
			\hline
		\end{tabular}	
	\end{table}

	\begin{figure}
		\centering
		\subfigure[$x_{0} = zero(100,1)$]{\includegraphics[height=5cm,width=6cm]{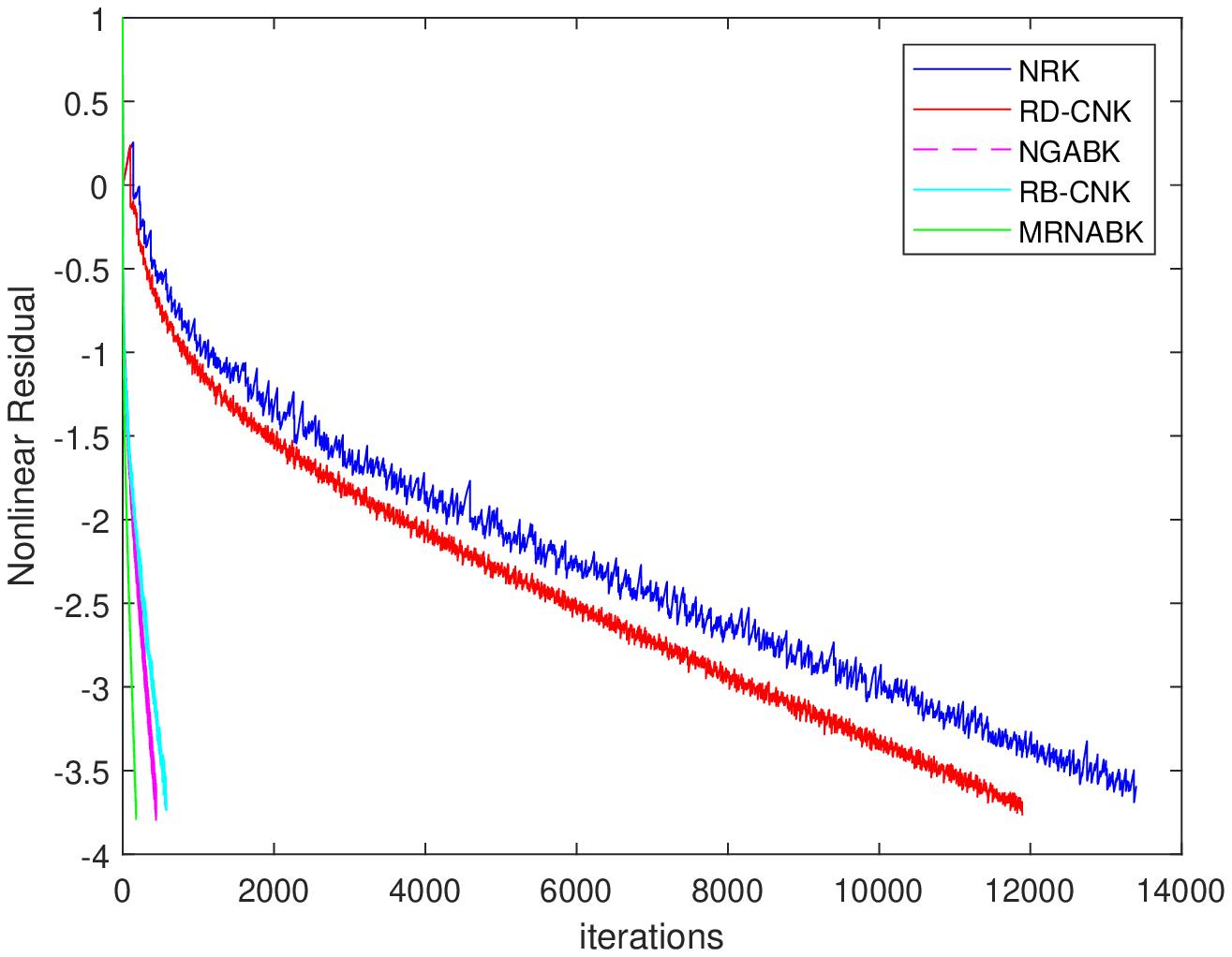}}
		\subfigure[$x_{0} = zero(300,1)$]{\includegraphics[height=5cm,width=6cm]{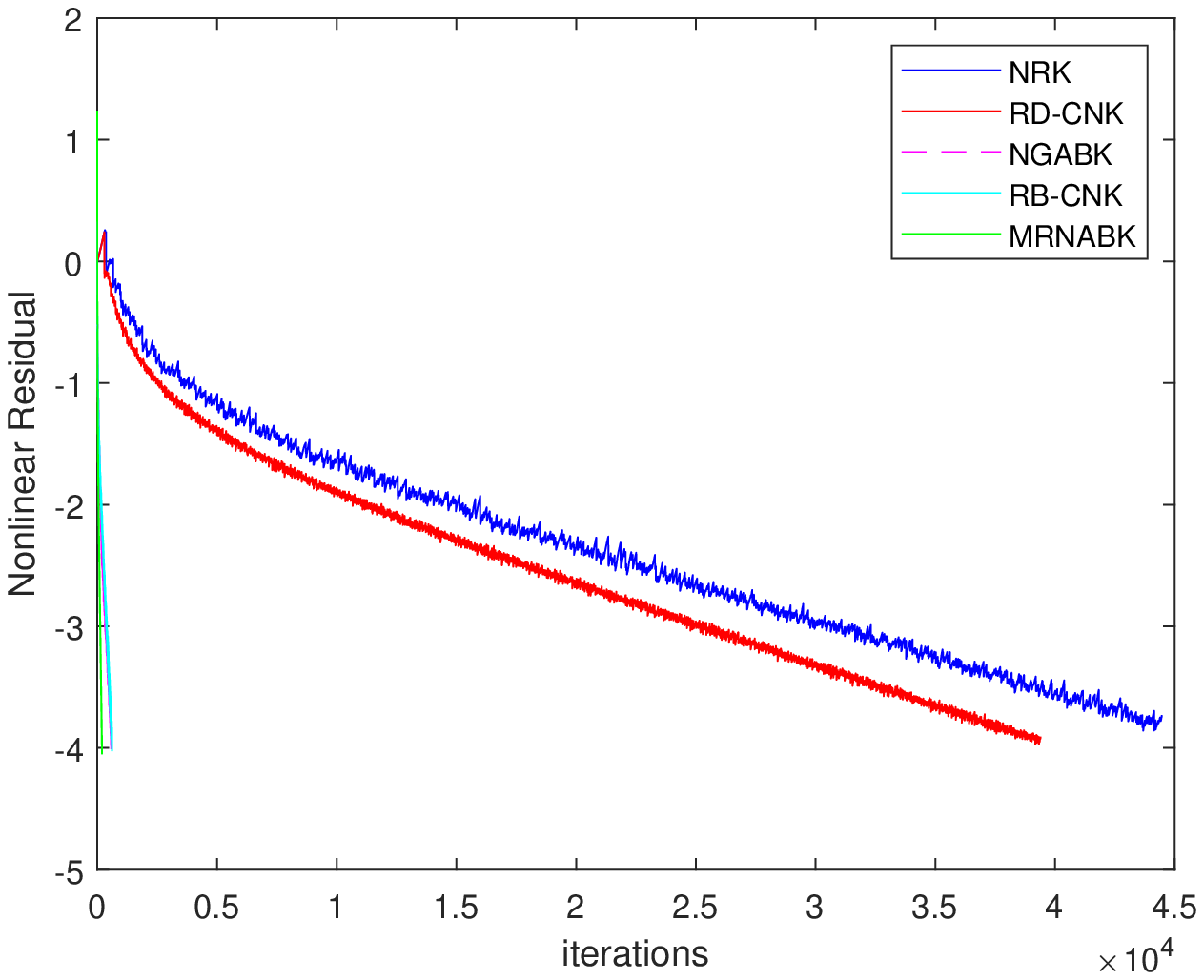}}
		\caption[d]{H-equation at different initial value.}
		\label{fig:1}
	\end{figure}
	
	\begin{figure}
		\centering
		\subfigure[CPU]{\includegraphics[height=5cm,width=6cm]{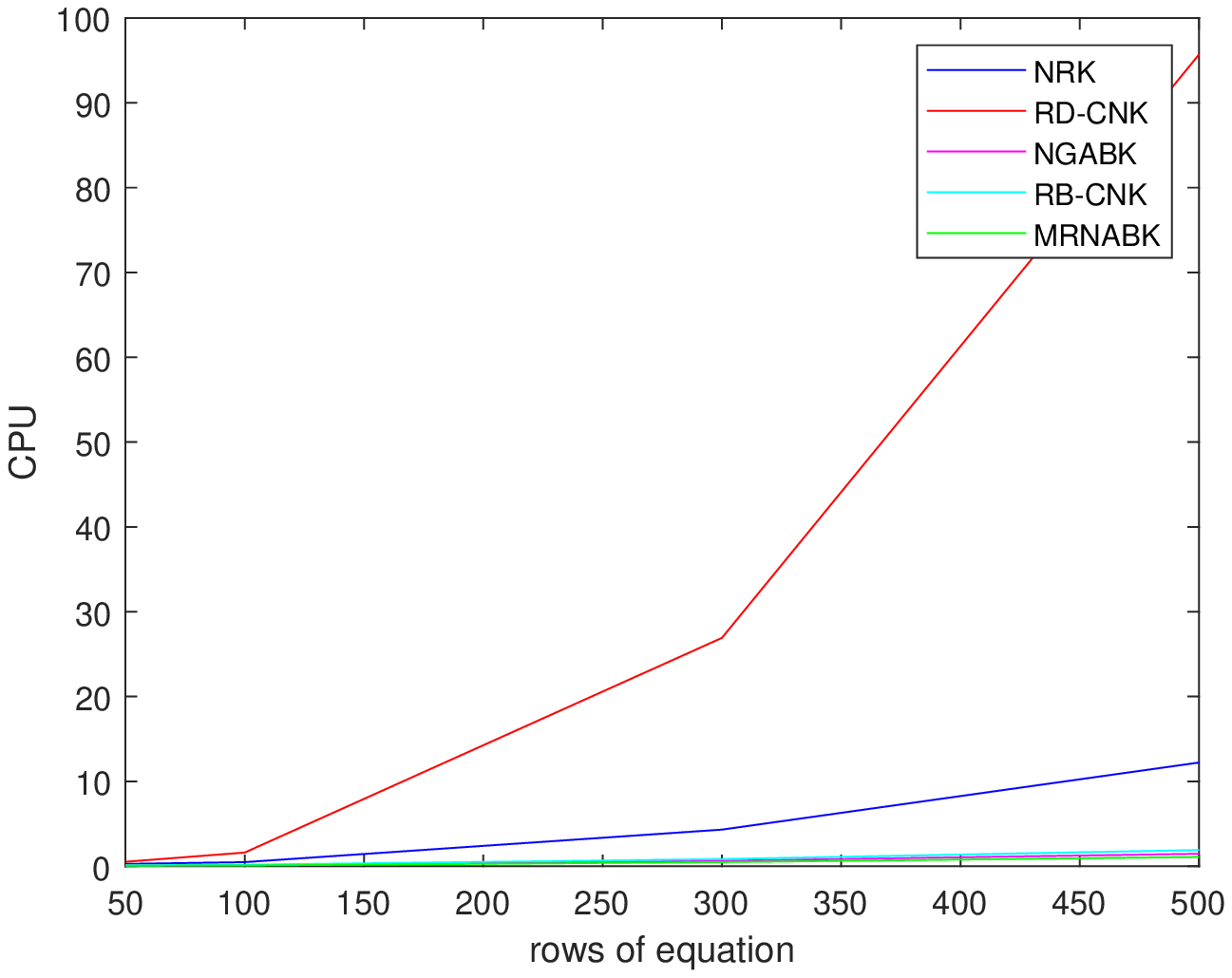}}
		\subfigure[IT]{\includegraphics[height=5cm,width=6cm]{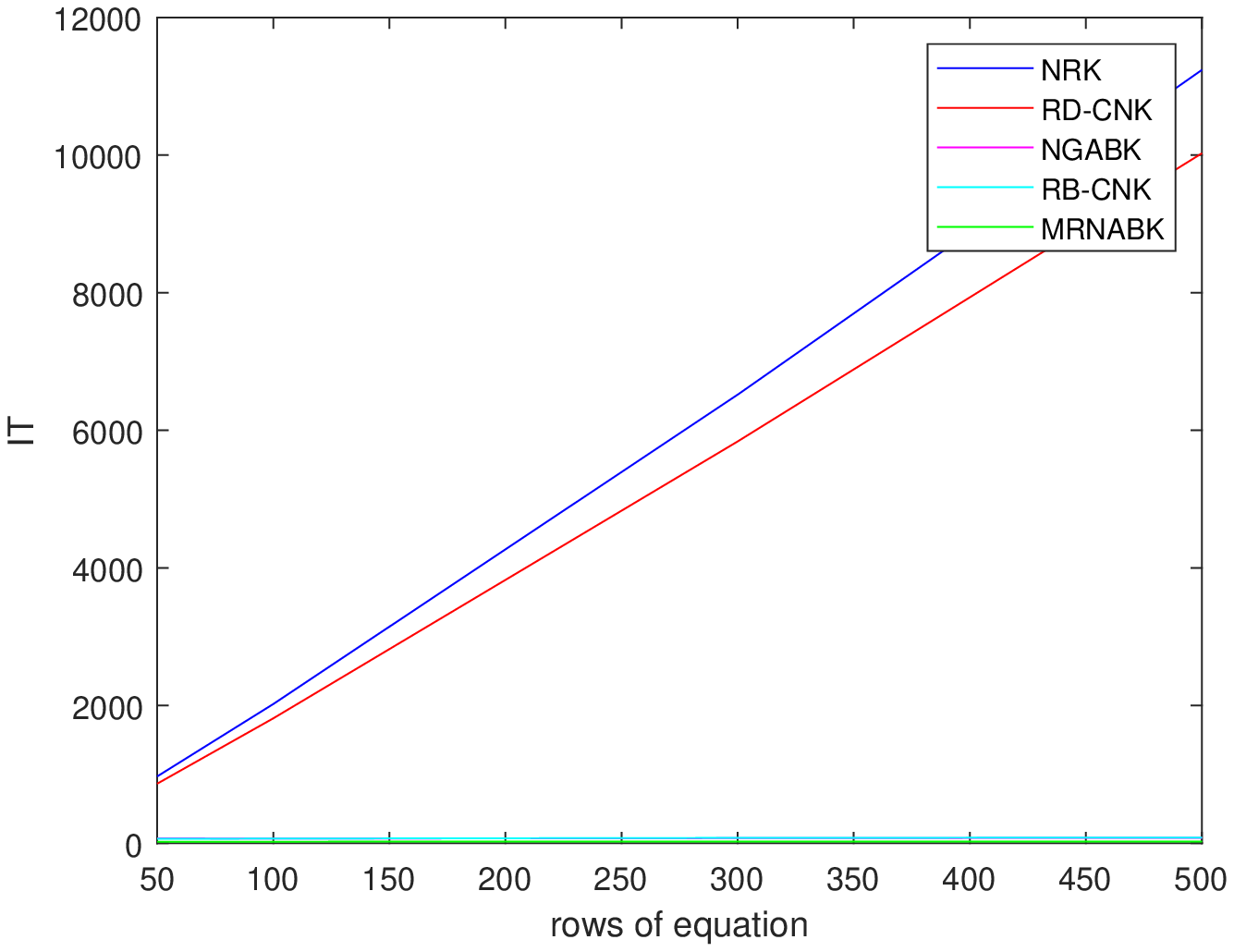}}
		\caption[d]{H-equation at different scales.}
		\label{fig:2}
	\end{figure}
	
\end{example}
\begin{example}
	\rm In this example, we consider the Brown almost linear function,
	\begin{equation}
		\begin{aligned}
		&f_{k}(x) = x^{(k)} + \sum_{i=1}^{n}x^{i} - (n+1),& 1\le k \le n; \\ \nonumber
    	&f_{k}(x) = \left( \prod \limits_{i=1}^n x^{(i)}\right) - 1, &k=n.\nonumber
    	\end{aligned}
    \end{equation}	
\indent In this experiment, we set the initial value $x_{0} = 0.5 \ast ones(n,1)$. The number of equations and the number of unknowns is set to $50\times50$, $100\times100$, $150\times150$, $200\times200$, $250\times250$, $300\times300$, $350\times350$, $400\times400$. We list the computing time and iteration numbers of these methods respectively in Table \ref{tab4} and Table \ref{tab5}. The results show that our new methods greatly outperform the NRK method. We observe that the iteration time of the NGABK method and the MRNABK method is almost the same in Table \ref{tab5}, but both of them are better than the RB-CNK method.
\begin{table}
	\centering
	\caption{IT comparison of NRK, RD-CNK, NGABK, RB-CNK, MRNABK}
	\label{tab4}
\begin{tabular}{cccccc}
	\hline
	$m \times n$& NRK& RD-CNK& NGABK& RB-CNK& MRNABK \\
	\hline
	$50 \times 50$& 4660& 755& 1& 1& 1\\
	$100 \times 100$& 15881& 1308& 1& 1& 1\\
	$150 \times 150$& 35398& 1904& 1& 1& 1\\
	$200 \times 200$& 58127& 2506& 1& 1& 1\\
	$250 \times 250$& 85937& 3128& 1& 1& 1\\
	$300 \times 300$& 116851& 3750& 1& 1& 1\\
	$350 \times 350$& 156027& 4372& 1& 1& 1\\
	$400 \times 400$& 196134& 4992& 1& 1& 1\\
	\hline
\end{tabular}	
\end{table}
\begin{table}
	\centering
	\caption{CPU comparison of NRK, RD-CNK, NGABK, RB-CNK, MRNABK}
	\label{tab5}
	\begin{tabular}{cccccc}
		\hline
		$m \times n$& NRK& RD-CNK& NGABK& RB-CNK& MRNABK \\
		\hline
		$50 \times 50$& 0.7222& 0.2290& 0.0024& 0.0049& 0.0017\\
		$100 \times 100$& 2.5929& 1.2108& 0.0050& 0.0063& 0.0045\\
		$150 \times 150$& 7.5696& 2.7863& 0.0108& 0.0149& 0.0112\\
		$200 \times 200$& 15.4563& 5.9908& 0.0187& 0.0250& 0.0188\\
		$250 \times 250$& 24.4239& 11.2057& 0.0321& 0.0446& 0.0320\\
		$300 \times 300$& 35.7111& 17.4006& 0.0630& 0.0737& 0.0530\\
		$350 \times 350$& 53.0089& 25.8307& 0.0508& 0.0775& 0.0476\\
		$400 \times 400$& 71.1793& 36.0714& 0.0626& 0.0982& 0.0699\\
		\hline
	\end{tabular}	
\end{table}
\end{example}

\begin{example}
	\rm In this example, we consider the following system of equations,
	\begin{equation}
		\begin{aligned}
		&f_{k}(x) = ((3-2x_{k})x_{k}-2x_{k+1} + 1)^{2},&k=1;\\
		&f_{k}(x) = ((3-2x_{k})x_{k}-x_{k-1}-2x_{k+1} + 1)^{2},&1<k<n;\\
		&f_{k}(x) = ((3-2x_{k})x_{k}-x_{k-1} + 1)^{2},&k=n.
		\nonumber
		\end{aligned}
	\end{equation}
\indent In this experiment, we set the initial value $x_{0} = (-0.5,-0.5,\dots,-0.5)^{T}$ and $n$ is the number of the equations. Singular Broyden problem is a square nonlinear system of equations, and its Jacobian matrix is singular at the solution. As can be seen from Table \ref{tab6} and Table \ref{tab7}, the NGABK method converges faster than the other three methods in terms of the number of iteration steps and calculation time. As can be seen from Fig. \ref{fig:3}, the residuals of MRNABK method decline the fastest, while those of NRK method decline the slowest. From Fig. \ref{fig:4}, we can observe that the approximate solutions can be obtained by all five methods.

\begin{table}
	\centering
	\caption{IT comparison of NRK, RD-CNK, NGABK, RB-CNK, MRNABK}
	\label{tab6}
	\begin{tabular}{cccccc}
		\hline
		$n$& NRK& RD-CNK& NGABK& RB-CNK& MRNABK\\
		\hline
		50& 1524& 1440& 288& 374& 33\\
		500& 18068& 17139& 4531& 6841& 33\\
		700& 25707& 24413& 4357& 10044& 34\\
		900& 33508& 31876& 4867& 13045& 33\\
		1500& 57340& 54586& 13502& 22743& 34\\
		2000& 77487& 73764& 12756& 22528& 31\\
		\hline
	\end{tabular}	
\end{table}

\begin{table}
	\centering
	\caption{CPU comparison of NRK, RD-CNK, NGABK, RB-CNK, MRNABK}
	\label{tab7}
	\begin{tabular}{cccccc}
		\hline
		$n$& NRK& RD-CNK& NGABK& RB-CNK& MRNABK\\
		\hline
		50& 0.9005& 0.4328& 0.1038& 0.1824& 0.0422\\
		500& 3.8890& 6.7895& 1.9450& 3.0060& 0.6040\\
		700& 6.0818& 12.7727& 3.1249& 4.8353& 0.8893\\
		900& 8.9058& 19.6451& 4.5572& 7.2238& 1.2652\\
		1500& 20.2697& 35.6259& 11.6104& 18.6458& 3.1587\\
		2000& 32.7492& 71.2359& 19.3939& 33.8771& 4.6496\\
		\hline
	\end{tabular}	
\end{table}

\begin{figure}
	\centering
	\subfigure[$n=500$]{\includegraphics[height=5cm,width=6cm]{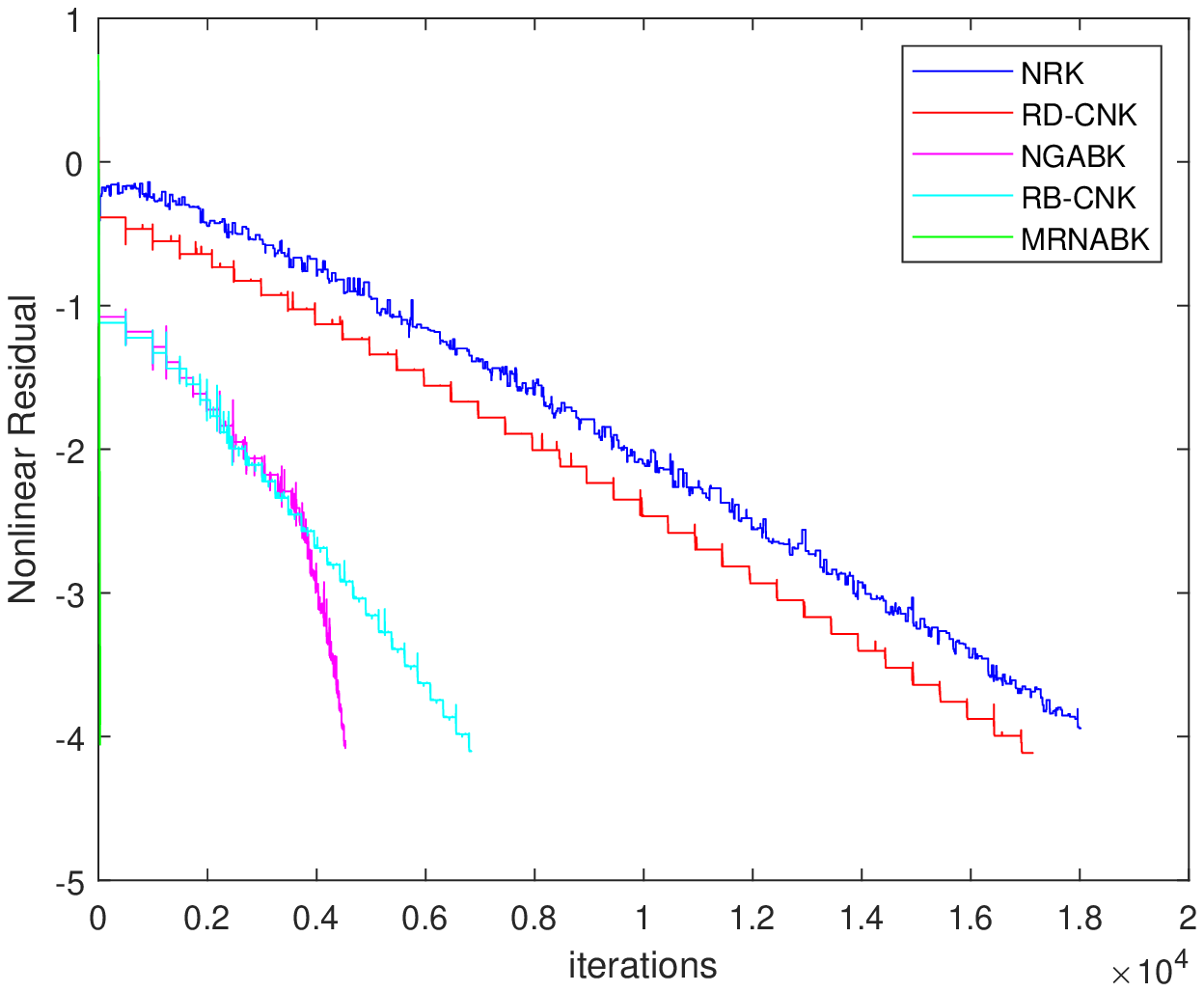}}
	\subfigure[$n=2000$]{\includegraphics[height=5cm,width=6cm]{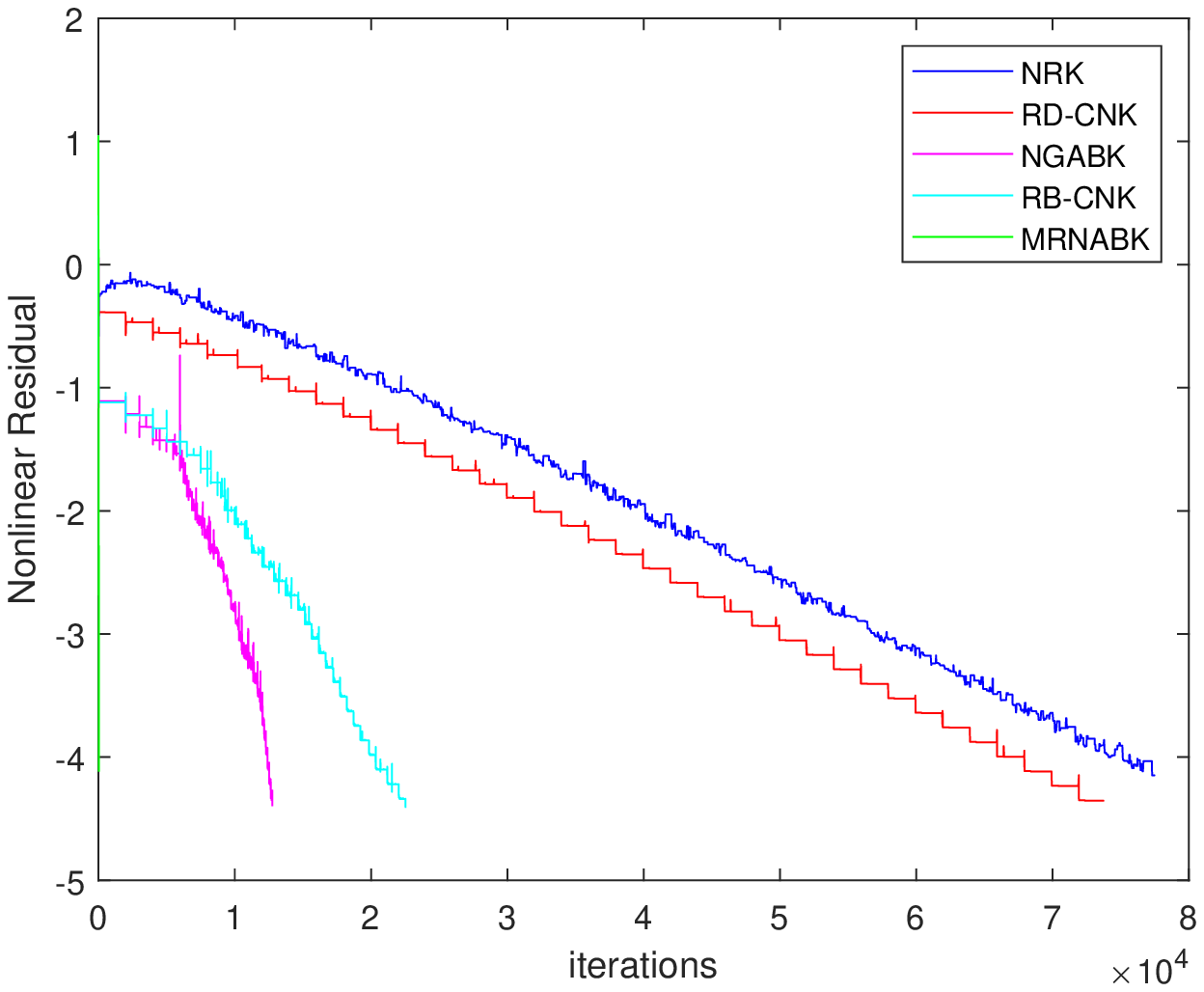}}
	\caption[d]{Singular Broyden problem with n = 500 (left), 2000 (right).}
	\label{fig:3}
\end{figure}

\begin{figure}
	\centering
	\subfigure[$n=500$]{\includegraphics[height=5cm,width=6cm]{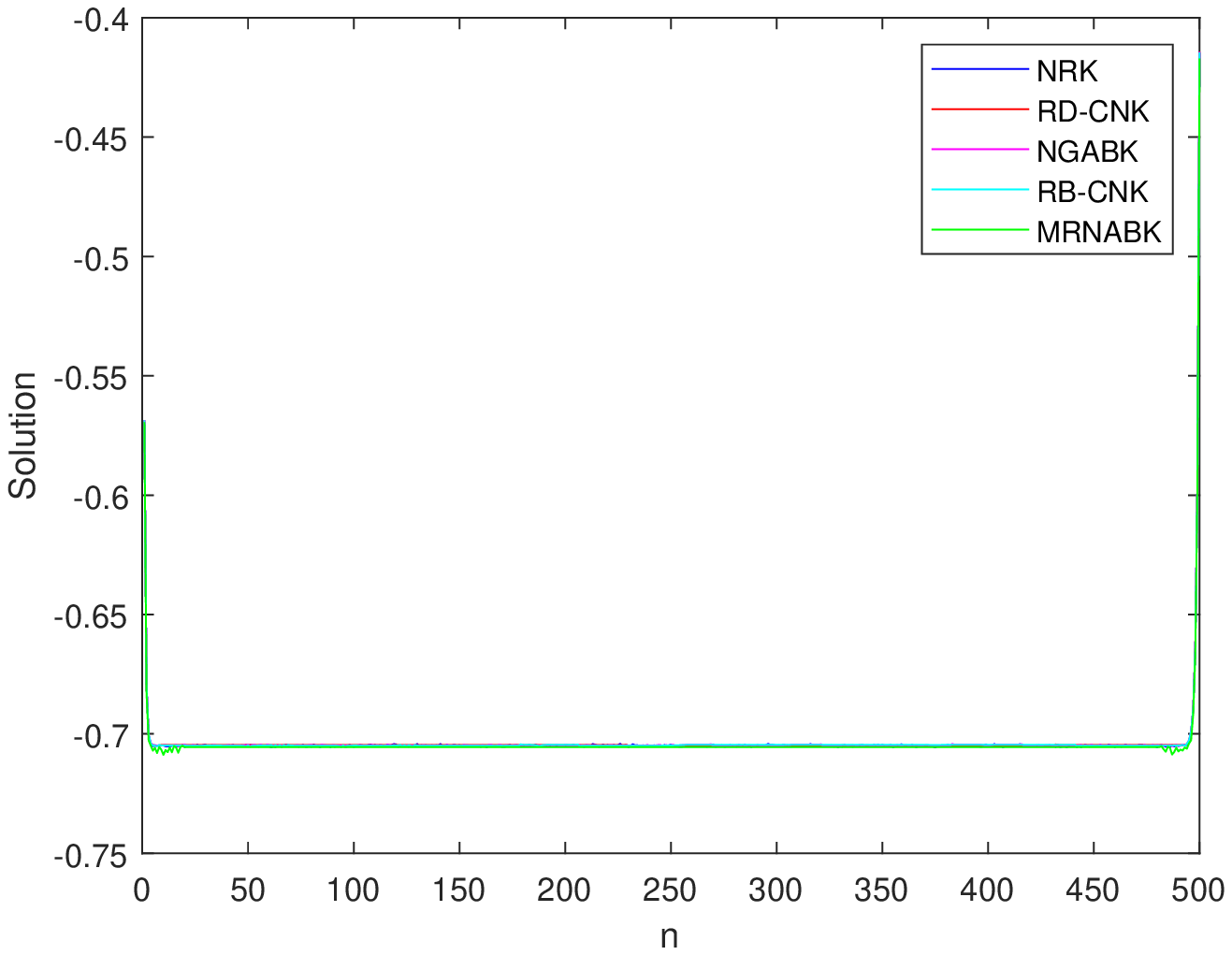}}
	\subfigure[$n=2000$]{\includegraphics[height=5cm,width=6cm]{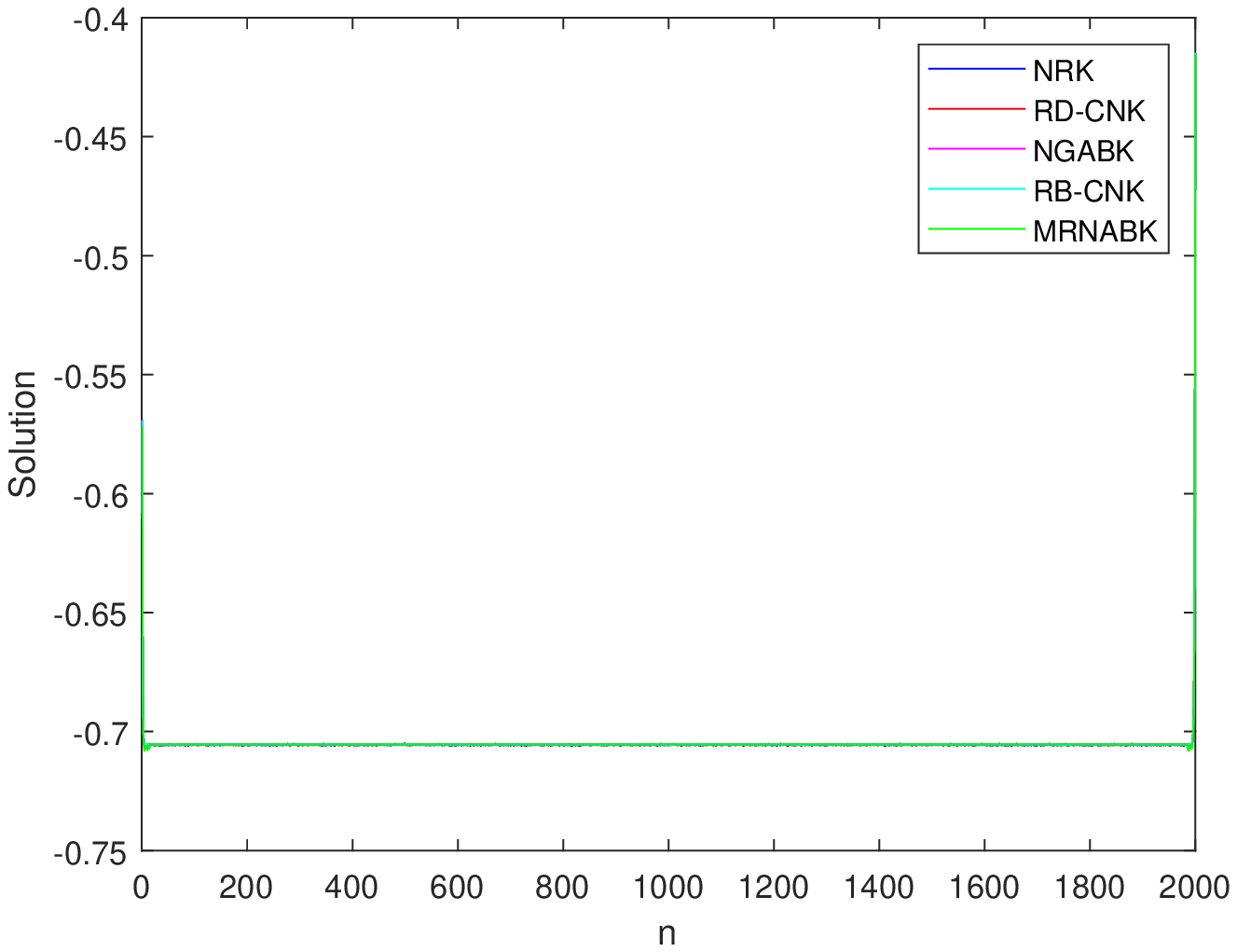}}
	\caption[d]{Singular Broyden problem's solution with n = 500 (left), 2000 (right).}
	\label{fig:4}
\end{figure}
\end{example}
\begin{example}
	\rm Consider the following overdetermined nonlinear problem,
	\begin{equation}
		\begin{aligned}
		&f_{k}(x) = 10\left(\frac{2x_{i}}{(1+(x_{i})^{2})^{2}}-x_{i+1}\right),&\quad mod(k,2)=1,\\
		&f_{k}(x) = x_{i} - 1,&\quad mod(k,2)=0,\\
		&m=2(n-1),&\quad i= div(k+1,2),
		\nonumber
		\end{aligned}
	\end{equation}
 where m represents the number of equations. In this experiment, we set the initial value $x_{0} =(0,0,\dots,0)$ and $n =$ 100, 300, 500, 1000, 2000, respectively. From Table \ref{tab8}, we can see that NRK achieves the worst numerical result compared with the other four methods. Further, we observe that RB-CNK, MRNABK and NGABK achieve the same result in terms of the number of iterations. From Table \ref{tab9}, NGABK and MRNABK are slightly faster than RB-CNK in terms of iteration time, so our approach is generally effective.
 \begin{table}
 	\centering
 	\caption{IT comparison of NRK, RD-CNK, NGABK, RB-CNK, MRNABK}
 	\label{tab8}
 	\begin{tabular}{cccccc}
 		\hline
 		$n$& NRK& RD-CNK& NGABK& RB-CNK& MRNABK\\
 		\hline
 		100& 138& 124& 2& 2& 2\\
 		300& 435& 328& 2& 2& 2\\
 		500& 706& 517& 2& 2& 2\\
 		1000& 1472& 1042& 2& 2& 2\\
 		2000& 2829& 2042& 2& 2& 2\\
 		\hline
 	\end{tabular}	
 \end{table}

\begin{table}
	\centering
	\caption{CPU comparison of NRK, RD-CNK, NGABK, RB-CNK, MRNABK}
	\label{tab9}
	\begin{tabular}{cccccc}
		\hline
		$n$& NRK& RD-CNK& NGABK& RB-CNK& MRNABK\\
		\hline
		100& 0.0186& 0.0482& 0.0094& 0.0131& 0.0085\\
		300& 0.1510& 0.2999& 0.0664& 0.0709& 0.0357\\
		500& 0.2273& 0.3888& 0.0896& 0.1357& 0.0836\\
		1000& 0.5298& 0.9602& 0.2230& 0.9057& 0.2019\\
		2000& 1.8300& 3.8600& 0.9893& 5.4136& 0.7435\\
		\hline
	\end{tabular}	
\end{table}
\end{example}
	\section{Conclusion}
	In this paper, based on the Gaussian Kaczmarz method and RD-CNK method, we propose a new class of nonlinear Kaczmarz  block methods to solve nonlinear equations and study their convergence theories. These methods use the average technique instead of calculating the Moore-Penrose pseudoinverse of the Jacobian matrix, which greatly reduces the amount of computation. Numerical results show that the NGABK method and the MRNABK method perform well in the case of singular Jacobian matrix and overdetermined equations. The study of the pseudoinverse-free method and the more efficient greedy rules is very meaningful, and this is what we need to continue to work on in the future.
	\bibliographystyle{unsrt}
	\bibliography{refenence}
\end{document}